\documentclass[11pt]{amsart}
\usepackage[letterpaper,margin=1in]{geometry}                
\usepackage{amssymb,amsmath,amsfonts,stmaryrd}
\usepackage{xcolor}
\usepackage[colorlinks,
             linkcolor=blue,
             citecolor=green!70!black,
             pdfproducer={pdfLaTeX},
             pdfpagemode=None,
             bookmarksopen=true
             bookmarksnumbered=true]{hyperref}
\usepackage{graphicx}
\usepackage{multicol}

\usepackage{tikz}
\usetikzlibrary{shapes.geometric,decorations.markings,arrows,decorations.pathreplacing,calc}
\usepackage{arydshln}
\usetikzlibrary{shapes.geometric,calc}
\usetikzlibrary{shapes,arrows,chains}
\usetikzlibrary{decorations.markings}
\usetikzlibrary{decorations.pathmorphing}
\usetikzlibrary{shapes.multipart}

\newcommand\arXiv[1]{\href{http://arxiv.org/abs/#1}{\nolinkurl{arXiv:#1}}}
\newcommand\MRnumber[1]{\href{http://www.ams.org/mathscinet-getitem?mr=#1}{\nolinkurl{MR#1}}}
\newcommand\DOI[1]{\href{http://dx.doi.org/#1}{\nolinkurl{DOI:#1}}}
\newcommand\MAILTO[1]{\href{mailto:#1}{\nolinkurl{#1}}}

\newcounter{mainthm}

\newtheorem{maintheorem}[mainthm]{Theorem}

\newtheorem{dummy}{Dummy}[section]

\newtheorem{lemma}[dummy]{Lemma}
\newtheorem{proposition}[dummy]{Proposition}
\newtheorem{corollary}[dummy]{Corollary}
\newtheorem{definition}[dummy]{Definition}

\theoremstyle{definition}
\newtheorem*{rem}{Remark}

\usepackage{dsfont}
\renewcommand\mathbb\mathds

\newcommand\bC{\mathbb C}

\newcommand\bZ{\mathbb Z}

\newcommand\cC{\mathcal C}

\newcommand\rH{\mathrm H}

\newcommand\rU{\mathrm U}

\newcommand\SH{\mathrm {SH}}
\newcommand\gp{\mathrm {gp}}

\usepackage[mathscr]{euscript}

\DeclareMathOperator\homology{H}
\renewcommand\H{\homology}

\newcommand\condense{\mathrel{\,\hspace{.75ex}\joinrel\rhook\joinrel\hspace{-.75ex}\joinrel\rightarrow}}

\newcommand\longto\longrightarrow
\newcommand\mono\hookrightarrow
\newcommand\epi\twoheadrightarrow

\newcommand\<\langle
\renewcommand\>\rangle
\newcommand\sminus\smallsetminus

\DeclareMathOperator\End{End}

\newcommand\define[1]{\emph{#1}}
\newcommand\cat[1]{\textsc{#1}}

\title{Fusion 2-categories with no line operators are grouplike}
\author[Johnson-Freyd and Yu]{Theo Johnson-Freyd$^{1,2}$ and Matthew Yu$^{*,1}$}
\thanks{ Research at the Perimeter Institute is supported by the Government of Canada through Industry Canada and by the Province of Ontario through the Ministry of Economic Development and Innovation. 
 The Perimeter Institute is in the Haldimand Tract, land promised to the Six Nations. Dalhousie University is in Mi`kma`ki, the ancestral and unceded territory of the Mi`kmaq. We are all Treaty people.
 \\[6pt]
$^1$ \textsc{Perimeter Institute for Theoretical Physics, Waterloo, Ontario}.\\ 
$^2$ \textsc{Department of Mathematics, Dalhousie University, Halifax, Nova Scotia}.
\\[6pt]
$^*$ Corresponding author. \MAILTO{myu@perimeterinstitute.ca}
}

\begin{document}
\begin{abstract}
We show that if $\cC$ is a fusion $2$-category in which the endomorphism category of the unit object is  $\cat{Vec}$ or $\cat{SVec}$, then the indecomposable objects of $\cC$ form a finite group.
\end{abstract}
\maketitle

\section{Introduction}

Just as multifusion $1$-categories describe the fusion of quasiparticle excitations ---
$1$-spacetime-dimensional objects, aka line operators --- 
in topological phases of matter, multifusion $2$-categories (first introduced in \cite{douglas2018fusion}) describe the fusion of $2$-spacetime-dimensional ``quasistring'' excitations, aka surface operators. Except in very low dimensions, a typical topological phase can have quasistring excitations which are not determined by the quasiparticle excitations, and multifusion $2$-categories are vital for the construction and classification of topological phases in medium dimension \cite{Lan_2017,PhysRevX.8.021074,johnsonfreyd2020classification}.

Recall that a multifusion $1$-category $\cC$ is \define{fusion} if the endomorphism algebra $\Omega\cC = \End_\cC(1_\cC)$ is trivial, i.e.\ isomorphic to $\bC$, where $1_\cC \in \cC$ denotes the monoidal unit \cite{etingof2015tensor}. There are two reasonable categorifications of this notion when $\cC$ is a multifusion $2$-category. The stronger generalization, which we will call \define{strongly fusion}, is to ask that the endomorphism $1$-category $\Omega\cC = \End_\cC(1_\cC)$ be trivial, i.e.\ equivalent to $\cat{Vec}_\bC$. The weaker notion, which we will call merely \define{fusion}, is to ask only that $\Omega^2\cC = \End_{\Omega\cC}(1_{1_\cC})$ be trivial, where $1_{1_\cC} \in \Omega\cC$ is the identity object. A \define{fusion 2-category} is a finite semisimple monoidal 2-category that has left and right duals for objects and a simple monoidal unit.  Physically, if $\cC$ describes the surface operators in a topological phase, then $\Omega\cC$ describes the line operators and $\Omega^2\cC$ describes the vertex ($0$-spacetime-dimensional) operators.

The classification of fusion $1$-categories is extremely rich \cite{etingof2004classification,Jordan_2009,natale2018classification}. The simplest examples are the \define{grouplike}, aka pointed, fusion $1$-categories, whose isomorphism classes of simple objects form a group $G$ under the fusion product. These are famously classified by ordinary group cohomology $\rH_{\gp}^3(G; \rU(1))$. But there are many nongrouplike examples.  The classification of (merely) fusion $2$-categories is similarly rich, since it includes the classification of braided fusion $1$-categories \cite[Construction 2.1.19]{douglas2018fusion}. The main result of this note shows a dramatic difference with the strongly fusion case:

\begin{maintheorem}\label{theorem-bosonic}
  If $\cC$ is a strongly fusion $2$-category, then the equivalence classes of indecomposable objects of $\cC$ form a finite group under the fusion product.
\end{maintheorem}

We also address the ``fermionic'' case where $\Omega\cC \cong \cat{SVec}$:

\begin{maintheorem}\label{theorem-super}
  If $\cC$ is a  fusion $2$-category
  with $\Omega\cC \cong \cat{SVec}$, then the equivalence classes of indecomposable objects of $\cC$ form a finite group, which is a central double cover of the group $\pi_0\cC$ of components of $\cC$ (see Definition~\ref{defn.component}).
\end{maintheorem}

In particular, Theorem~\ref{theorem-super} asserts that the components of $\cC$ do form a group.

\begin{rem}
  Just as grouplike fusion $1$-categories in which the simple objects form a group $G$ are classified by $\H^3_{\gp}(G;\rU(1))$, the strongly fusion $2$-categories with simple objects $G$ are classified by $\rH_{\gp}^4(G;\rU(1))$ \cite[Remark 2.1.17]{douglas2018fusion}. In the fermionic case,
  if one additionally assumes that the actions of $\End(1_\cC) \cong \cat{SVec}$ on $\End(X)$ given by tensoring on the left and on the right agree, then
  one can show that the options with $\pi_0\cC = G$ are classified by ``extended group supercohomology'' $\SH_{\gp}^4(G)$ defined in \cite{Wang:2017moj}. There is a canonical map $\SH_{\gp}^4(G) \to \H_{\gp}^2(G; \bZ_2)$ which takes an extended supercohomology class to its \define{Majorana layer}; the group of simple objects in $\cC$ is the corresponding central extension $\bZ_2.G$. Although in general Majorana layers of supercohomology classes have no reason to be trivial, we were unable to find an example where the extension $\bZ_2.G$ did not split.
\end{rem}

The outline of our paper is as follows. Section~\ref{section-DR} reviews the definition of multifusion $2$-category from \cite{douglas2018fusion}. In particular, we recall their notion of ``component'' of a semisimple $2$-category in \S\ref{section-DR}. The proofs of Theorems~\ref{theorem-bosonic} and~\ref{theorem-super} occupy \S \ref{grouplike} and \S\ref{SVecCase}, respectively.

In future work, we will use these theorems to give a complete classification of $5$-spacetime-dimensional topological orders.

\section{Semisimple and multifusion 2-categories} \label{section-DR}

The definition and basic theory of semisimple and multifusion $2$-categories were first introduced in \cite{douglas2018fusion}. Since this theory is new, we take this section to review the main features.

Recall that a $2$-category $\cC$ is \define{$\bC$-linear} if all hom-sets of $2$-morphisms are vector spaces over $\bC$, and both $1$- and $2$-categorical compositions of $2$-morphisms are bilinear.


\begin{definition}
  An object in a linear 2-category is \define{decomposible} if it is equivalent to a direct sum of nonzero objects, and \define{indecomposable} if it is nonzero and not decomposable.
\end{definition}

\begin{rem}
We will slightly abuse the language and use the terms ``simple'' and ``indecomposable'' interchangeably. A \define{simple} object $X$ in a $2$-category is one such that any faithful $1$-morphism $A \hookrightarrow X$ is either $0$ or an equivalence. In finite semisimple $2$-categories all indecomposable objects are simple \cite{douglas2018fusion}.
\end{rem}

In particular the objects which we consider in the 2-category will only be sums of finitely many simple objects, and decompositions are unique up to permutations.  In our goal to define a semisimple 2-category, we present some definitions for the higher categorical generalization of the notion of idempotent splitting and idempotent complete for 1-categories, also discussed in \cite{gaiotto2019condensations}.

\begin{definition}
  A 2-category $\cC$  is  \define{locally idempotent complete}  if for all objects $A, B \in \cC$, the 1-category $\hom_\cC(A,B)$ is idempotent complete. It is \define{locally finite semisimple} if $\hom_\cC(A,B)$ is furthermore a finite semisimple $\bC$-linear category (i.e.\ an abelian $\bC$-linear category with finitely many isomorphism classes of simple object and in which every object decomposes as a finite direct sum of simple objects).
\end{definition}
In what follows, we will assume $\cC$ is a locally idempotent complete 2-category.  
\begin{definition}
  A \define{separable monad} is a unital algebra object $p \in \hom_\cC(A,A)$, for a simple object $A$, whose multiplication $m: p \circ p \to p$ admits a section as a $p$-$p$ bimodule. 
\end{definition}

\begin{definition}
  A \define{(unital) condensation} in a $2$-category $\cC$ is an  adjunction $f \dashv g \equiv (f : A \leftrightarrows B : g, \eta : 1_A \to g\circ f, \epsilon : f\circ g \to 1_B)$  which is \define{separable} in the sense that its counit $\epsilon$ admits a section, i.e.\ if there exists a $2$-morphsism $\phi : 1_B \to f\circ g$ which is the right inverse of $\epsilon$. 
  When there is such a separable adjuction, we will write ``$A \condense B$,'' and say that \define{$A$ condenses onto $B$}.
\end{definition}

\begin{definition}
  A separable monad $p$ is \define{separably split} if there exists a separable adjunction $f \dashv g$ and $g \circ f \cong p$.  A \define{separable splitting} is a choice of this isomorphism.
\end{definition}

\begin{proposition}[{\cite[Proposition 1.3.4]{douglas2018fusion}}]
   A separable monad in $\cC$ which admits a separable splitting, admits a unique up-to-equivalence separable splitting. 
\end{proposition}
Admitting a separable splitting implies that the adjunction $f \dashv g$ admits $A$ as an Eilenberg-Moore object.  In 1-categories, this is can be seen as module decomposition by forming a projector from an idempotent.  The subtlety in 2-categories is that now there is no ``orthogonal complement" to the projector, as in 1-categories.   

\begin{definition}\label{2IdempotentComplete}
  A 2-category $\cC$ is \define{2-idempotent complete} if it is locally idempotent complete and every separable monad splits.  
\end{definition}

\begin{rem}
Requiring the unitality of $p$ and the existence of a unit for adjunction in the 2-category case differs slightly from the situation in 1-categories.  In 1-categories there is an equality of $p^2=p$ but there is no equality of $1$ and $p$.  \cite{gaiotto2019condensations} developed a nonunital version of separable monad for 2-categories and showed that if $\cC$ has adjoints for 1-morphisms, then the notion of 2-idempotent completion in Definition~\ref{2IdempotentComplete} and in \cite{gaiotto2019condensations} agree.
\end{rem}

\begin{definition}\label{defn.semisimple2cat}
  A $\bC$-linear $2$-category $\cC$ is \define{finite semisimple} if: it has finitely many isomorphism classes of simple objects;
  it is locally finite semisimple; has adjoints for $1$-morphisms; has direct sums of objects; and is $2$-idempotent complete.
\end{definition}

\begin{definition}
  A \define{multifusion $2$-category} is a monoidal finite semisimple $2$-category in which all objects have duals.
\end{definition}

\begin{rem}
  As noted in \cite[Definition 2.1.6]{douglas2018fusion}, in a fusion 2-category, left and right duals are the same.
\end{rem}

The $1$-categorical Schur's Lemma says that in a semisimple $1$-category, if two indecomposable objects are related by a nonzero morphism, then they are isomorphic. This result fails in $2$-categories, but \cite[Proposition 1.2.19]{douglas2018fusion} provides the following replacement.

\begin{proposition}[Categorical Schur's Lemma]
  If $f : X\to Y$ and $g : Y \to Z$ are nonzero $1$-morphisms between indecomposable objects in a semisimple $2$-category $\cC$, then $gf : X \to Z$ is nonzero.
\end{proposition}
\begin{proof}
 This follows from Proposition~\ref{condense} below, since the composition of condensations is a condensation and since condensations with nonzero target are nonzero.
\end{proof}

In particular, ``related by a nonzero morphism'' defines an equivalence relation on the indecomposable objects of $\cC$. (Note that, since every $1$-morphism is required to have an adjoint, if there is a nonzero morphism $f : X \to Y$, then there is a nonzero morphism $f^* : Y \to X$.)

\begin{definition} \label{defn.component}
  The set of \define{components} of $\cC$, denoted $\pi_0 \cC$, is the set of equivalence classes of indecomposable objects for the equivalence relation ``related by a nonzero morphism.''
\end{definition}

The structure of each component is fully determined by (the endomorphism category of) any representative object. Indeed:

\begin{proposition}\label{condense}
  Suppose $X,Y \in \cC$ are simple objects connected by a nonzero $1$-morphism $f : X \to Y$. Then there is a condensation $X \condense Y$. In particular, $Y$ is the image of a simple algebra object in the fusion $1$-category $\End_\cC(X)$. 
 \end{proposition}

\begin{proof}
Choose $g$ to be the right adjoint to $f$; it exists because all morphisms in a semisimple $2$-category are required to have adjoints. The counit $\epsilon : f\circ g \to 1_Y$ is a nonzero $1$-morphism in the semisimple $1$-category $\End_\cC(Y)$ with simple target, and so has a section.
%
\end{proof}

\section{Proof of Theorem~\ref{theorem-bosonic}}\label{grouplike}

We begin this section by developing the necessary graphical calculus in order to prove the main results.  One important feature we will discuss is the state-operator map for fusion  2-categories and its interplay with duality.  
For an object $X\in \cC$, we denote $\int_{S^1_b}X$ as the wrapping of $X$ around a boundary-framed $S^1$, see Figure \ref{wrapS1}. 
This integral is a map $\int_{S^1_b}: \cC \to \Omega \cC $.  This integral is an example of the general calculus of dualizability as in \cite{freed2020gapped}, and also arising from the cobordism hypothesis \cite{Baez_1995}. 
\begin{figure}[t]

\begin{tikzpicture}

\centering
      \def\Depth{4}
        \def\Height{2}
        \def\Width{2}
        \coordinate (O) at (-3,0+1,0-1);
        \coordinate (A) at (-3,\Width+1.5,0-1);
        \coordinate (B) at (-3,\Width+1.5,\Height);
        \coordinate (C) at (-3,0+1,\Height);
         \coordinate (W) at (-3,0+1,0-1);
        \coordinate (X) at (-3,\Width+1.5,0-1);
        \coordinate (Y) at (-3,\Width+1.5,\Height);
        \coordinate (Z) at (-3,0+1,\Height);
\draw[  left color=white, white] (O)--(A)--(B)--(C)--cycle;
\draw[  left color=blue!30,
  right color=blue!60,
  middle color=red!20, black] (W)--(X)--(Y)--(Z)--cycle;
  \draw[black] (-3,2.1,0.5) node{$X$};
   \draw[->,decorate,decoration={snake,amplitude=.4mm,segment length=2mm,post length=1mm}] (-1.75,2.5,1) -- (0,2.5, 1);
   \draw[black] (-.9,2.9,1) node{$\int_{S^1_b}$};

\centering
\node [draw,
  shape=cylinder,
  name=nodename, 
  alias=cyl, 
  aspect=1.5,
  minimum height=3cm,
  minimum width=2cm,
  left color=blue!30,
  right color=blue!60,
  middle color=red!20, 
  outer sep=-0.5\pgflinewidth, 
  shape border rotate=90
] at (1,2) {$X$};

\fill [black!0.5] let
  \p1 = ($(cyl.before top)!0.5!(cyl.after top)$),
  \p2 = (cyl.top),
  \p3 = (cyl.before top),
  \n1={veclen(\x3-\x1,\y3-\y1)},
  \n2={veclen(\x2-\x1,\y2-\y1)}
 in 
  (\p1) ellipse (\n1 and \n2);

\end{tikzpicture}
\caption{}
\label{wrapS1}
\end{figure}
The boundary framing of the cylinder in Figure \ref{wrapS1} is attained from the framing of the annulus, where the annulus framing is given by the restriction of the two-dimensional blackboard framing, see Figure \ref{AnnulusFraming}. One could then take the framed annulus and pull the annulus into a cylinder.  This results in a cylinder appropriately framed to be compatible with the state-operator map.

\begin{figure}[b]
    \begin{tikzpicture}
    \draw[left color=blue!30,
  right color=blue!60,
  middle color=red!20]  (2,0) circle (2);
  \draw[left color=white!20,
  right color=white!20,
  middle color=white!20] (2,0) circle (1);
    \draw[->,decorate, line width = .3mm] (.5,0) -- (.75,0);
    \draw[->,decorate,line width = .3mm] (.5,0) -- (.5,.25);
    \draw[->,decorate, line width = .3mm] (0.939340,1.06066) -- (0.939340+.25,1.06066);
    \draw[->,decorate,line width = .3mm] (0.939340,1.06066) -- (0.939340,1.06066+.25);
    \draw[->,decorate, line width = .3mm] (0.939340,-1.06066) -- (0.939340+.25,-1.06066);
    \draw[->,decorate,line width = .3mm] (0.939340,-1.06066) -- (0.939340,-1.06066+.25);
    \draw[->,decorate, line width = .3mm] (3.06066,1.06066) -- (3.06066+.25,1.06066);
    \draw[->,decorate,line width = .3mm] (3.06066,1.06066) -- (3.06066,1.06066+.25);
    \draw[->,decorate, line width = .3mm] (3.06066,-1.06066) -- (3.06066+.25,-1.06066);
    \draw[->,decorate,line width = .3mm] (3.06066,-1.06066) -- (3.06066,-0.81066);
    \draw[->,decorate,line width = .3mm] (3.5,0) -- (3.75,0);
    \draw[->,decorate,line width = .3mm] (3.5,0) -- (3.5,.25);
    \draw[->,decorate,line width = .3mm]  (2,-1.5)--(2.25,-1.5);
    \draw[->,decorate, line width = .3mm] (2,-1.5) -- (2,-1.25);
    \draw[->,decorate,line width = .3mm] (2,1.5) -- (2.25,1.5);
    \draw[->,decorate, line width = .3mm] (2,1.5) -- (2,1.75);
    \end{tikzpicture}
    \caption{}
    \label{AnnulusFraming}
\end{figure}

We now describe this operation $\int_{S^1_b}$ algebraically.
Because we are working with a fusion 2-category each object has a dual and we have a unit $\eta_X: 1_{\cC} \to X \otimes X^*$.  
It corresponds to the half-circle with framing as in Figure~\ref{2rings} part~(a).
Also, since all 1-morphisms have adjoints, there is a right adjoint $\eta^*_X:  X \otimes X^* \to 1_{\cC} $, corresponding to the framed half-circle in Figure~\ref{2rings} part~(b). These two half-circles compose to an annulus whose framing can be continuously deformed to framing in Figure~\ref{AnnulusFraming}. All together, we find the algebraic definition:
$$ \int_{S^1_b} X :=  \eta^*_X \circ  \eta_X.$$
%
%
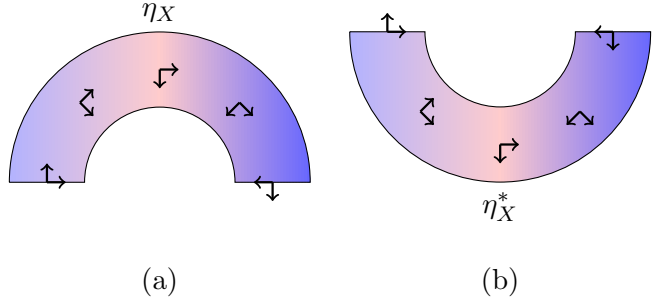
\begin{figure}[t]
    \begin{tikzpicture}
    \draw[left color=blue!30,
  right color=blue!60,
  middle color=red!20] (0,0) -- (1,0) arc (180:0:1) -- (3,0) -- (4,0) arc (0:180:2);
    \draw[->,decorate,line width = .3mm] (2,1.5) -- (2.25,1.5);
    \draw[->,decorate, line width = .3mm] (2,1.5) -- (2,1.25);
    \draw[->,decorate, line width = .3mm] (0.939340,1.06066) -- (0.939340+0.178,1.06066-0.178);
    \draw[->,decorate,line width = .3mm] (0.939340,1.06066) -- (0.939340+0.178,1.06066+0.178);
    \draw[->,decorate, line width = .3mm] (.5,0) -- (.75,0);
    \draw[->,decorate,line width = .3mm] (.5,0) -- (.5,.25);
    
    \draw[->,decorate, line width = .3mm] (3.06066,1.06066) -- (3.06066+0.178,1.06066-0.178);
    \draw[->,decorate,line width = .3mm] (3.06066,1.06066) -- (3.06066-0.178,1.06066-0.178);
    
    \draw[->,decorate,line width = .3mm] (3.5,0) -- (3.25,0);
    \draw[->,decorate,line width = .3mm] (3.5,0) -- (3.5,-.25);
    \draw[above] (2,2) node{$\eta_X$};
    \draw[below]  (2,-1) node{(a)};
    \end{tikzpicture}
    \quad 
    \begin{tikzpicture}
    \draw[left color=blue!30,
  right color=blue!60,
  middle color=red!20] (0,0) --(1,0) arc (180:360:1) -- (3,0) -- (4,0) arc (0:-180:2); 
    \draw[->,decorate, line width = .3mm] (.5,0) -- (.75,0);
    \draw[->,decorate,line width = .3mm] (.5,0) -- (.5,.25);
    \draw[->,decorate,line width = .3mm] (3.5,0) -- (3.25,0);
    \draw[->,decorate,line width = .3mm] (3.5,0) -- (3.5,-.25);

    \draw[->,decorate, line width = .3mm] (0.939340,-1.06066) -- (0.939340+0.178,-1.06066-0.178);
    \draw[->,decorate,line width = .3mm] (0.939340,-1.06066) -- (0.939340+0.178,-1.06066+0.178);
    
    \draw[->,decorate, line width = .3mm] (3.06066,-1.06066) -- (3.06066+0.178,-1.06066-0.178);
    \draw[->,decorate,line width = .3mm] (3.06066,-1.06066) -- (3.06066-0.178,-1.06066-0.178);

    \draw[->,decorate,line width = .3mm]  (2,-1.5)--(2.25,-1.5);
    \draw[->,decorate, line width = .3mm] (2,-1.5) -- (2,-1.75);
    \draw[below] (2,-2) node{$\eta^*_X$};
    \draw[below] (2,-3) node{(b)};
    \end{tikzpicture}
    \caption{$\eta^*_X$ is by definition the universal map such that the composition with $\eta_X$ can be filled.  The resulting framing of $\eta^*_X \circ \eta_X$ is homotopic to the blackboard framing of Figure~\ref{AnnulusFraming}. }
    \label{2rings}
\end{figure}

The vertex operators of $X$ are by definition the $2$-morphisms $1_X \Rightarrow 1_X$. They are precisely the operators that can be inserted in the interior hole in Figure~\ref{AnnulusFraming}; the blackboard framing is arranged so that this can happen.
Such an insertion may be pulled down and thought of as a map $1_{1_\cC} \to \int_{S_1^b} X$ as in Figure~\ref{stateMap}. In other words, $\hom(1_{1_{\cC}},\int_{S^1_b} X )$ is the vector space of ways for the vacuum line to end on $\int_{S^1_b} X$. This is the physical/geometric proof of the \define{state-operator correspondence}. 
\begin{figure}[b]
\begin{tikzpicture}

\centering
      \def\Depth{4}
        \def\Height{2}
        \def\Width{2}
        \coordinate (O) at (1,0+1,0-1);
        \coordinate (A) at (1,\Width+1.5,0-1);
        \coordinate (B) at (1,\Width+1.5,\Height);
        \coordinate (C) at (1,0+1,\Height);
         \coordinate (W) at (-3,0+1,0-1);
        \coordinate (X) at (-3,\Width+1.5,0-1);
        \coordinate (Y) at (-3,\Width+1.5,\Height);
        \coordinate (Z) at (-3,0+1,\Height);
\draw[  left color=white, white] (O)--(A)--(B)--(C)--cycle;
\node [draw,
  shape=cylinder,
  name=nodename, 
  alias=cyl, 
  aspect=1.5,
  minimum height=3cm,
  minimum width=2cm,
  left color=blue!30,
  right color=blue!60,
  middle color=red!20, 
  outer sep=-0.5\pgflinewidth, 
  shape border rotate=90
] at (1,2) {$X$};

\fill [black!0.5] let
  \p1 = ($(cyl.before top)!0.5!(cyl.after top)$),
  \p2 = (cyl.top),
  \p3 = (cyl.before top),
  \n1={veclen(\x3-\x1,\y3-\y1)},
  \n2={veclen(\x2-\x1,\y2-\y1)}
 in 
  (\p1) ellipse (\n1 and \n2);
  
  \coordinate (X) at (1,.70);
  \coordinate (Y) at  (1,-0);
  \draw[dotted ,line width = .7 mm, opacity = .4] (X) -- (Y);
\draw (1,-.3) node {$1_{1_{\cC}}$};
\end{tikzpicture}
\caption{}
\label{stateMap}
\end{figure}
Algebraically, we have:
\begin{lemma}[State-Operator Correspondence]\label{stateoperatormap}
In a {multifusion 2-category} there is an isomorphism $\End_{\End_{\cC}(X)}(1_{X}) \cong \hom_{\Omega\cC}(1_{1_{\cC}}, \int_{S^1_b} X )$. 
\end{lemma}
\begin{proof}
The duality of $X$ with $X^*$ provides an equivalence of $\End_{\cC}(X) \cong \hom(1_{\cC}, X \otimes X^*)$. This equivalence identifies $1_X$ with $\eta_X$, and so in particular $\End(1_{X}) \cong \End(\eta_X)$, where the left-hand side is computed in $\End_{\cC}(X)$ and the right-hand side is computed in $\hom(1_{\cC}, X \otimes X^*)$. For any adjunctible $1$-morphism $f : A \to B$ in a $2$-category, $\End_{\hom(A,B)}(f) \cong \hom_{\End A}(1_{1_A}, f^* \circ f)$. Taking $f = \eta_X$, with $A = 1_\cC$ and $B = X \otimes X^*$, completes the proof.
\end{proof}

 
 In particular, $X$ is simple if and only if $\hom_{\Omega\cC}(1_{1_\cC}, \int_{S^1_b}X)$ is one-dimensional.
 In the strongly fusion case lemma \ref{stateoperatormap} implies:



\begin{proposition}\label{SimpleX}
  Suppose $\cC$ is strongly fusion. Then $X \in \cC$ is indecomposable if and only if $\int_{S^1_b}X = \bC$. \qed
\end{proposition}

We now consider the tensor product of two indecomposable objects $X\otimes Y$ mapped by the integral $\int_{S^1_{b}}$. This represents a cylinder within a cylinder as on the left of Figure \ref{RemoveInner}.

\begin{figure}[hbt!]
\begin{tikzpicture}\label{cylinders}
\node [draw,
  shape=cylinder,
  name=nodename, 
  alias=cyl, 
  aspect=1.5,
  minimum height=3cm,
  minimum width=3cm,
  left color=blue!30,
  right color=blue!60,
  middle color=red!20, 
  outer sep=-0.5\pgflinewidth, 
  shape border rotate=90
] at (1,2) {\phantom{A}};

\fill [black!0.5] let
  \p1 = ($(cyl.before top)!0.5!(cyl.after top)$),
  \p2 = (cyl.top),
  \p3 = (cyl.before top),
  \n1={veclen(\x3-\x1,\y3-\y1)},
  \n2={veclen(\x2-\x1,\y2-\y1)}
 in 
  (\p1) ellipse (\n1 and \n2);
  
    \node [draw,
  shape=cylinder,
  name=nodename2, 
  alias=cyl2, 
  aspect=0.7,
  minimum height=2.8cm,
  minimum width=1.7cm,
  left color=blue!30,
  right color=blue!60,
  middle color=red!20, 
  outer sep=-0.5\pgflinewidth, 
  shape border rotate=90,opacity=.5,
] at (1,2) {\phantom{A}};

\fill [black!0.5] let
  \p1 = ($(cyl2.before top)!0.5!(cyl2.after top)$),
  \p2 = (cyl2.top),
  \p3 = (cyl2.before top),
  \n1={veclen(\x3-\x1,\y3-\y1)},
  \n2={veclen(\x2-\x1,\y2-\y1)}
 in 
  (\p1) ellipse (\n1 and \n2);
   \coordinate (X) at (1,.7);
  \coordinate (Y) at  (1,-.3);
  \draw[dotted ,line width = .7 mm, opacity = 0] (X) -- (Y);
  
    \draw[->,decorate,decoration={snake,amplitude=.4mm,segment length=2mm,post length=1mm}] (3.5,2.5,1) -- (5,2.5, 1);

\end{tikzpicture}
\quad 
\begin{tikzpicture}
    \node [draw,
  shape=cylinder,
  name=nodenamenew, 
  alias=cylnew, 
  aspect=1.5,
  minimum height=3cm,
  minimum width=3cm,
  left color=blue!30,
  right color=blue!60,
  middle color=red!20, 
  outer sep=-0.5\pgflinewidth, 
  shape border rotate=90
] at (1,2) {\phantom{A}};

\fill [black!0.5] let
  \p1 = ($(cylnew.before top)!0.5!(cylnew.after top)$),
  \p2 = (cylnew.top),
  \p3 = (cylnew.before top),
  \n1={veclen(\x3-\x1,\y3-\y1)},
  \n2={veclen(\x2-\x1,\y2-\y1)}
 in 
  (\p1) ellipse (\n1 and \n2);
  
  \coordinate (X) at (1,.7);
  \coordinate (Y) at  (1,0);
  \draw[dotted ,line width = .7 mm, opacity = .4] (X) -- (Y);
  \draw (1,-.3) node {$1_{1_{\cC}}$};
\end{tikzpicture}
\caption{}
\label{RemoveInner}
\end{figure}
In general, we see that $\int_{S^1_b}$ is not monoidal: a cylinder within a cylinder is not the same as two adjacent cylinders. However, in the strongly fusion case, 
if $X$ and $Y$ are simple then 
we may collapse down the inner cylinder via the state operator map into the vacuum line. We may then collapse the outer cylinder. All together we find:
%
%
%
%
\begin{corollary}\label{cor-pi0ismonoid}
  In a strongly fusion 2-category, the tensor product of indecomposable objects is indecomposable. \qed
\end{corollary}

This allows us to complete the proof of Theorem~\ref{theorem-bosonic}:

\begin{proof}[Proof of Theorem~\ref{theorem-bosonic}]
  If $X \in \cC$ is a simple object, then $X^*$ is as well (since $\End(X) \cong \End(X^*)$), and hence so is $X \otimes X^*$ (by Corollary~\ref{cor-pi0ismonoid}). Since $\eta_X : 1_\cC \to X \otimes X^*$ is nonzero, the simple objects $1_\cC$ and $X \otimes X^*$ are in the same component. However, the fact that there are no lines in the strongly fusion case means that $1_\cC$ is the only simple object in its component.
\end{proof}


\begin{rem}
We can consider working over the real numbers, which is the same as having an anti-linear involution (time-reversal).  In this case, an indecomposable object is not absolutely simple, and Theorem~\ref{theorem-bosonic} is no longer true.  We can see this already at the level of fusion 1-categories. Consider a $\bZ_3$ fusion 1-category with three objects $\{ 1, x, x^{-1}\}$ over $\bC$.  Over the real numbers, we can exchange $x$ and $x^{-1}$
by the involution.  There will be two objects $1$ and $X$ over the real numbers, where $X \cong x + x^{-1}$, so that it is invariant under the involution. Schur's lemma states that over the complex numbers, indecomposable means that the endomorphisms of the object is just $\bC$.  But over the real numbers, the endomorphisms are a division ring, and we have the fusion $X^2 = X + 2$.
\end{rem}

\section{Proof of Theorem~\ref{theorem-super}}\label{SVecCase}

If we try to repeat the proof from \S\ref{grouplike} when $\Omega\cC = \End(1_\cC) \cong \cat{SVec}$, the first snag arises in Proposition~\ref{SimpleX}. 
Indecomposability of $X$ implies that the ordinary vector space $\hom(1_{1_\cC}, \int_{S^1_b}X  )$ is one-dimensional. This measures the even part of the super vector space $\int_{S^1_b}X$, but says nothing about the odd part. 
On the other hand, the super vector space $\int_{S^1_b} X$ is the superalgebra of vertex operators on $X$, and so it is supercommutative because we have the freedom to move operators around each other on the surface of the cylinder. Furthermore, since we are working in a semisimple $2$-category, this supercommutative algebra is finite dimensional and semisimple.

\begin{lemma}
  The only finite-dimensional semisimple supercommutative superalgebra $A$ with one-dimensional bosonic part is $\bC$.
\end{lemma}
\begin{proof}
  If $x \in A$ is a nonzero odd element, then it is nilpotent (since $x^2 = xx = -xx$ and so $x^2 = 0$). Thus the principle ideal generated by $x$ is proper.
  On the other hand, it is not a direct summand of $A$ as an $A$-module because projection operators are bosonic, and so the only projection operators in $A$ are $0$ and $1$.
  This contradicts the semisimplicty of $A$.
\end{proof}

This implies the fermionic versions of Proposition~\ref{SimpleX} and Corollary~\ref{cor-pi0ismonoid}. 

To complete the proof of Theorem~\ref{theorem-super}, it suffices to observe that if $X \in \cC$ is indecomposable, then, since the unit map $\eta_X : 1_\cC \to X \otimes X^*$ is nonzero, $X \otimes X^*$ is an indecomposable object in the identity component of $\cC$, and so invertible, and thus $X$ is invertible. Indeed, since $\Omega\cC \cong \cat{SVec}$, by Proposition~\ref{condense} there are precisely two simple objects in the identity component of $\cC$, corresponding to the two simple superalgebras $\bC$ and $\mathrm{Cliff}(1)$, and $\mathrm{Cliff}(1)$ is famously Morita-invertible.

\begin{rem}
  In fact, $X \otimes X^*$ is always trivial, and never the nontrivial simple objected $\mathrm{Cliff}(1)$. Indeed, it is a general fact of monoidal higher categories that if an object is invertible, then its inverse is its dual. One can also see this directly by running the proof of Proposition~\ref{condense} for the nonzero 1-morphism $f = \eta_X$. Then $g = \eta_X^*$, and the simple algebra in question is the composition $p = gf = \eta_X^* \circ \eta_X = \int_{S^1_b}X = \bC$.
\end{rem}


\bibliography{F2C}{}
\bibliographystyle{alpha}

\end{document}